\documentclass[10pt]{article}
\usepackage[dvips]{graphics}
\usepackage{amsfonts}
\usepackage{amsmath}
\usepackage{algorithmic}
\usepackage{amssymb}
\usepackage{amsthm}
\usepackage{graphicx}
\usepackage{hyperref}
\usepackage[all,cmtip]{xy}

\makeatletter
\def\oversortoftilde#1{\mathop{\vbox{\m@th\ialign{##\crcr\noalign{\kern3\p@}%
      \sortoftildefill\crcr\noalign{\kern3\p@\nointerlineskip}%
      $\hfil\displaystyle{#1}\hfil$\crcr}}}\limits}

\def\sortoftildefill{$\m@th \setbox\z@\hbox{$\braceld$}%
  \braceld\leaders\vrule \@height\ht\z@ \@depth\z@\hfill\braceru$}

\makeatother

\usepackage[english]{babel}
\usepackage{color}
\usepackage{setspace}
\usepackage{marvosym}
\usepackage{geometry}

\newtheorem{theorem}{Theorem}[section]
\newtheorem{proposition}[theorem]{Proposition}
\newtheorem{lemma}[theorem]{Lemma}

\newtheorem{definition}[theorem]{Definition}
\newtheorem{remark}[theorem]{Remark}
\newtheorem{example}[theorem]{Example}

\DeclareMathOperator{\car}{char}

\DeclareMathOperator{\diam}{diam}

\newtheorem*{tmaa*}{Theorem A}
\newtheorem*{tmab*}{Theorem B}
\newtheorem*{tmac*}{Theorem C}
\newtheorem*{tmad*}{Theorem D}
\newtheorem*{tmae*}{Theorem E}
\newtheorem*{defi*}{Definition}

\title{Wandering domains for non-archimedean quadratic rational functions}

\author{V\'ictor Nopal-Coello}

\begin{document}

\maketitle
\begin{abstract}
Let $\mathbb{C}_K$ be a complete and algebraic closed non-archimedean field with residual characteristic \\
$\car(\widetilde{\mathbb{C}}_K)=2$. In this paper we prove that there exist $a,b\in\mathbb{C}_K$ such that the rational function $R(z)=\frac{z^2-z}{bz-\frac{1}{a}}$ has wandering components in its Fatou set. 
\end{abstract}

\section*{Introduction}\label{Sec:1}

Let $\mathbb{C}_K$ be a complete and algebraic closed non-archimedean field. Given a rational function $R\in\mathbb{C}_K(z)$ we can consider the dynamical system defined by the iterates of $R$ in the projective space $\mathbb{P}(\mathbb{C}_K)$. This dynamical system gives a partition of the projective space into two sets which are invariant under $R$: the {\it Fatou} set or the stable set and the {\it Julia} set or the unstable set.

In his Ph.D. thesis, Rivera-Letelier has shown that the Fatou set has two types of periodic components: the attracting basins and indifferent components. In turn, attracting basins are either disks or Cantor type components, whereas the indifferent components can be disks ({\it Rivera disk}) or can be of the form $$\mathbb{P}(\mathbb{C}_K)-\bigcup_{i=0}^m B_i,$$where the $B_i$'s are closed balls. We called this components {\it $m$-Rivera domains}. 

The Fatou set also can contains components which are not preperiodic, these components are called {\it wandering components}.  In \cite{RLBwd}, Benedetto shows that if the residue characteristic of $\mathbb{C}_K$ is $\car(\widetilde{\mathbb{C}}_K)=p>0$ then there exists $a\in\mathbb{C}_K$ such that the Fatou set of the polynomial function $\phi(z)=(1-a)z^{p+1}+az^p$ has wandering components. 

Note that the Benedetto's examples are functions with degree $\deg(\phi)=p+1\geq3$, leaving open the question over the existence of wandering components for quadratic rational functions.  

In this paper we prove there exist rational functions $R\in\mathbb{C}_k(z)$ of degree two with wandering components in its Fatou set, under the hypothesis that $\car(\widetilde{\mathbb{C}}_K)=2$. In fact, we prove the following theorem.

\begin{tmaa*}
Let $\mathbb{C}_K$ be a complete and algebraically closed non-archimedean field with residual characteristic $\car(\widetilde{\mathbb{C}}_K)=2$. Then there exist constants $a,b\in\mathbb{C}_K$ such that the rational function $$R(z)=\frac{z^2-z}{bz-\frac{1}{a}}$$has wandering components in its Fatou set.
\end{tmaa*}
We prove this theorem following similar ideas than Benedetto in \cite{RLBwd}. 

We divide this work in two sections. 
\begin{itemize}
\item In Section \ref{Sec:2} we give a brief account about quadratic rational functions and the dynamics generated by these functions, we refer the interested reader to \cite{NCVT} for a deeper analysis about quadratic rational functions.
\item In Section \ref{SecW} we give the necessary tools to prove Theorem A and then we prove this theorem. 
\end{itemize}

\newcounter{ns}
\section{Preliminaries}
\label{Sec:2}

We  start with a brief account on the basics of non-archimedean fields and introduce notation. 

\subsection{Disks, balls and spheres}

Let $\mathbb{C}_K$ be a complete and algebraically closed non-archimedean field. Let $a\in\mathbb{C}_k$ and $r>0$, then we define the disks and balls with center in $a$ and radius $r$ as follows $$D_r(a)=\{z\in\mathbb{C}_K:|z-a|<r\}\hspace{5mm}\text{and}\hspace{5mm}B_r(a)=\{z\in\mathbb{C}_K:|z-a|\leq r\}.$$

The non-archimedean property implies that every point of a disk (or ball) is a center, that is, if $b\in D_r(a)$ then $D_r(b)=D_r(a)$. Analogously, if $b\in B_r(a)$ then $B_r(b)=B_r(a)$. This implies that if two disks (or balls) has nonempty intersection then one contains the other.

Lemma 2.1 in \cite{JRLT} implies that if $R\in\mathbb{C}_K(z)$ is a rational function and $D_r(a)\in\mathbb{C}_K$ is a disk, then $R(D_r(a))$ is another disk or $R(D_r(a))=\mathbb{P}(\mathbb{C}_K)$, in particular if $R$ does not have poles in $D_r(a)$ then $R(D_r(a))$ is another disk.  

For $a\in\mathbb{C}_K$ and $r>0$, also we define the sphere with center in $a$ and radius $r$ as $$\mathbb{S}_r(a)=\{z\in\mathbb{C}_k:|z-a|=r\}.$$Note that $\mathbb{S}_r(a)\cap D_r(a)=\emptyset$ and $B_r(a)=\mathbb{S}_r(a)\cup D_r(a)$.

\subsection{Quadratic rational functions}

Let $R\in\mathbb{C}_K(z)$ be quadratic rational function. In \cite{NCVT} we have proved the existence of a trichotomy for the dynamics of $R$, that is, if $R$ is a quadratic rational function then
\begin{itemize}
    \item $R$ is simple (there exist an automorphism $\varphi$ such that the composition $\varphi\circ R\circ\varphi^{-1}$ has good reduction). In this case the Julia set is an empty set, or
    \item the Fatou set of $R$ consist only of an attracting basin which is of Cantor type. In this case the Julia set is nonempty, or
    \item the Fatou set contains an $m$-Rivera domain (for some $m\geq1$) and its Julia set is nonempty.
\end{itemize}

In the two first cases, the Fatou set does not contains wandering components, therefore we focus on only in the third case.

Let $R\in\mathbb{C}_K(z)$ be a rational function, following the notation given in \cite{NCVT}, if the Fatou set of $R$ contains an $m$-Rivera domain then we can write $R$ as 
\begin{equation}
\label{R}
R(z)=\frac{z^2-z}{bz-\frac{1}{a}},
\end{equation}
with $a,b\in\mathbb{C}_K$, $|a|>1$ and $|b|=1$. Note that $R(1)=0$ and $0$ is a fixed point with multiplier $a$, therefore $0$ is a repelling fixed point; and $\infty$ is a fixed point with multiplier $b$, therefore $\infty$ is an indifference fixed point. 

Let $c_1,c_2\in\mathbb{C}_K$ be the critical points of $R$. It is easy to check that $c_1=\frac{1+\sqrt{1-ab}}{ab}$ and $c_2=\frac{1-\sqrt{1-ab}}{ab}$, and therefore $|c_1|=|c_2|=\frac{1}{\sqrt{|a|}}$. Note that if $\car(\widetilde{\mathbb{C}}_K)>2$ then $|c_1-c_2|=\frac{1}{\sqrt{|a|}}$, while if $\car(\widetilde{\mathbb{C}}_K)=2$ then $|c_1-c_2|=\frac{1}{2\sqrt{|a|}}$.

By Lemma 3.9 in \cite{NCVT} we know the dynamics of $R$ in $\mathbb{P}(\mathbb{C}_K)$, in fact we have that 
\begin{itemize}
    \item if $D_r(x)\subset D_\frac{1}{|a|}(0)$ then $R(D_r(x))=D_{|a|r}(R(x))$,
    \item if $D_r(x)\subset D_\frac{1}{\sqrt{|a|}}(1)$ then $R(D_r(x))=D_{r}(R(x))$, and
    \item in $D_\frac{1}{\sqrt{|a|}}(c_1)$, $R$ is contracting. 
\end{itemize}

Moreover, we have the following lemma.

\begin{lemma}[Lemma 3.9, in \cite{NCVT}]
\label{LR}
Let $x,y\in\mathbb{C}_K$ and let $R\in\mathbb{C}_K(z)$ be a rational function as in \eqref{R}.
\begin{enumerate}
\item If $|x|,|y|>\frac{1}{\sqrt{|a|}}$ then $$|R(x)-R(y)|=|x-y|.$$ Therefore if $D_r(x)\subset\mathbb{C}_K-B_{1/\sqrt{|a|}}(0)$ then $R(D_r(x))=D_r(R(x))$, i.e., $R$ is an isometry.
\item If $|x|,|y|\leq\frac{1}{|a|}$ and $x,y\notin D_{1/|a|}(\frac{1}{ab})$ then $$|R(x)-R(y)|=|a||x-y|,$$ hence in $D_{1/|a|}(0)$, $R$ is expanding. Moreover if $D_r(x)\subset D_{1/|a|}(0)$ then $R(D_r(x))=D_{|a|r}(R(x))$.
\item If $\frac{1}{|a|}<|x|\leq\frac{1}{\sqrt{|a|}}$, $y\in D_{|x|}(x)$ and $x\notin D_\frac{1}{\sqrt{|a|}}(c_1)\cup D_\frac{1}{\sqrt{|a|}}(c_2)$ then $$|R(x)-R(y)|=\frac{|x-y|}{|a||xy|}.$$Therefore if $r\leq|x|$ then $R(D_r(x))=D_\delta(R(x))$ where $\delta=\frac{r}{|a||x|^2}$. In particular if $|x|=\frac{1}{\sqrt{|a|}}$ then $|R(x)-R(y)|=|x-y|$, hence $R(D_{\frac{1}{\sqrt{|a|}}}(x))=D_{\frac{1}{\sqrt{|a|}}}(R(x))$, that is $R$ is an isometry.
\item If $x,y\in D_\frac{1}{\sqrt{|a|}}(c_1)$ or $x,y\in D_\frac{1}{\sqrt{|a|}}(c_2)$ then $$|R(x)-R(y)|\leq|a||x-y|\max\{|x-c_1||x-c_2|,|x-y||bx-\frac{1}{a}|\}.$$In particular $|R(x)-R(y)|<|x-y|$, i.e. $R$ is contracting. But nevertheless $$R\left(D_{\frac{1}{\sqrt{|a|}}}(c_1)\right)=D_{\frac{1}{\sqrt{|a|}}}(R(c_1))\text{ and }R\left(D_{\frac{1}{\sqrt{|a|}}}(c_2)\right)=D_{\frac{1}{\sqrt{|a|}}}(R(c_2)).$$
\end{enumerate}
\end{lemma}

Further, if $|x-c_1||x-c_2|\neq|x-y||bx-\frac{1}{a}|$ then $$|R(x)-R(y)|=|a||x-y|\max\{|x-c_1||x-c_2|,|x-y||bx-\frac{1}{a}|\}.$$

\section{Wandering domains}
\label{SecW}
Let $\mathbb{C}_K$ be a complete and algebraically closed non-archimedean  field with %field such that its residue field $\widetilde{\mathbb{C}}_K$ is algebraic over a finite field. Assume that $\car(\widetilde{\mathbb{C}}_K)=2$.
residual characteristic $\car(\widetilde{\mathbb{C}}_K)=2$.

Let $a\in\mathbb{C}_K$ be a fixed number such that $|a|>1$. For each parameter $b\in\mathbb{C}_K$, with $|b|=1$, let us define the rational function $$R_b(z)=\frac{z^2-z}{bz-\frac{1}{a}}.$$For each rational function $R_b(z)$ let us denote its critical points by $c_1^b$ and $c_2^b$, and by $w_1^b$ and $w_2^b$ the preimages of $1$.

The following lemma give us the distance between the images of the same point $z$ under two different functions $R_b$ and $R_\beta$ with $|b|=|\beta|=1$.  

\begin{lemma}
\label{RemarkRb}
Let $R_b(z)=\frac{z^2-z}{bz-\frac{1}{a}}$ and $R_\beta(z)=\frac{z^2-z}{\beta z-\frac{1}{a}}$ quadratic rational functions with $|b|=|\beta|=1$ and $|a|>1$, then
\begin{enumerate}
\item if $|z|=\frac{1}{\sqrt{|a|}}$ then $|R_b(z)-R_\beta(z)|=|b-\beta|$.
\item if $z\in D_\frac{1}{\sqrt{2}}(1)$ then $|R_b(z)-R_\beta(z)|=|z-1||b-\beta|$.
\item if $|z|<\frac{1}{|a|}$ then $|R_b(z)-R_\beta(z)|=|a|^2|z|^2|b-\beta|$.
\end{enumerate}
\end{lemma}
\begin{proof}
Note that
\begin{align*}
|R_b(z)-R_\beta(z)|&=\left|\frac{z^2-z}{bz-\frac{1}{a}}-\frac{z^2-z}{\beta z-\frac{1}{a}}\right|\\
   &=|z||z-1|\frac{|\beta z-\frac{1}{a}-bz+\frac{1}{a}|}{|(bz-\frac{1}{a})(\beta z-\frac{1}{a})|}\\
   &=\frac{|z|^2|z-1||b-\beta|}{|bz-\frac{1}{a}||\beta z-\frac{1}{a}|}.
\end{align*}
Therefore, 
\begin{enumerate}
\item if $|z|=\frac{1}{\sqrt{|a|}}$ then $|z|^2=\frac{1}{|a|}$, $|z-1|=1$, $\left|bz-\frac{1}{a}\right|=\left|\beta z-\frac{1}{a}\right|=\frac{1}{\sqrt{|a|}}$ and hence $$|R_b(z)-R_\beta(z)|=\frac{\frac{1}{|a|}\cdot1\cdot|b-\beta|}{\frac{1}{\sqrt{|a|}}\cdot\frac{1}{\sqrt{|a|}}}=|b-\beta|.$$
\item if $z\in D_\frac{1}{\sqrt{2}}(1)$ then $|z|^2=1$, $\left|bz-\frac{1}{a}\right|=\left|\beta z-\frac{1}{a}\right|=1$ and hence $$|R_b(z)-R_\beta(z)|=\frac{1\cdot|z-1||b-\beta|}{1\cdot1}=|z-1||b-\beta|.$$
\item if $|z|<\frac{1}{|a|}$ then $|z-1|=1$, $\left|bz-\frac{1}{a}\right|=\left|\beta z-\frac{1}{a}\right|=\frac{1}{|a|}$ and hence $$|R_b(z)-R_\beta(z)|=\frac{|z|^2\cdot1\cdot|b-\beta|}{\frac{1}{|a|}\cdot\frac{1}{|a|}}=|a|^2|z|^2|b-\beta|.$$
\end{enumerate}
\end{proof}

In order to find a quadratic rational function with wandering components, we start with a quadratic rational function $R_{b_1}$ such that the ball $B_\frac{1}{2\sqrt{2}}(c_1^{b_1})$ is periodic of period $3$. Then, we vary the parameter $b_1$ (appropriately) until we find a suitable parameter $b$ such that the quadratic rational function $R_b(z)$ contains wandering components in its Fatou set. 

Note that, by Lemma \ref{LR}, it is enough to prove that $R_{b_1}^3(c_1^{b_1})\in B_{\frac{1}{2\sqrt{2}}}(c_1^{b_1})$ to conclude that $$R_{b_1}^3\left(B_{\frac{1}{2\sqrt{2}}}(c_1^{b_1})\right)=B_{\frac{1}{2\sqrt{2}}}(c_1^{b_1}).$$

\begin{example}
If $b_1=-1$ and $a=-1+\frac{\sqrt{-1}}{2}$, then $|b|=1$, $|a|=2$ and $$R_{b_1}(z)=\frac{z^2-z}{b_1z-\frac{1}{a}}=\frac{z^2-z}{-z-\frac{2}{-2+\sqrt{-1}}}.$$Note that a critical point of $R$ is $$c_1^{b_1}=1+\sqrt{-1},$$and $|c_1^b-c_2^b|=\frac{1}{2\sqrt{2}}$. 

It is easy to check that 
\begin{itemize}
    \item $R_{b_1}(c_1^{b_1})=-1-2\sqrt{-1}$,
    \item $R^2_{b_1}(c_1^{b_1})=\frac{6}{5}+\frac{26\sqrt{-1}}{15}$,
    \item $R^3_{b_1}(c_1^{b_1})=-\frac{599}{545}-\frac{3929\sqrt{-1}}{1635}$,
\end{itemize}
then $$|R^3_{b_1}(c_1^{b_1})-c_1^{b_1}|=\left|-\frac{599}{545}-\frac{3929\sqrt{-1}}{1635}-1-\sqrt{-1}\right|=\left|-\frac{1144}{545}-\frac{5564}{1635}\sqrt{-1}\right|=\frac{1}{4}<\frac{1}{2\sqrt{2}}.$$

By 1., 2., and 4. in Lemma \ref{LR} we have that $R^3_{b_1}(B_\frac{1}{2\sqrt{2}}(c_1^b))=B_\frac{1}{2\sqrt{2}}(R^3_{b_1}(c_1^{b_1}))$ and since $|R^3_{b_1}(c_1^b)-c_1^{b_1}|<\frac{1}{2\sqrt{2}}$ then $$R^3_{b_1}(B_\frac{1}{2\sqrt{2}}(c_1^{b_1}))=B_\frac{1}{2\sqrt{2}}(c_1^{b_1}).$$
$\hfill\square$
\end{example}

In order to facilitate the operations, we start with a quadratic rational function
\begin{equation}
    \label{Ecb1}
    R_{b_1}(z)=\frac{z^2-z}{b_1z-\frac{1}{a}}\in\mathbb{C}_K(z)   
\end{equation}
with $|b_1|=1$, $|a|=2$, such that $\displaystyle{R^3_{b_1}\left(B_{\frac{1}{2\sqrt{2}}}(c_1^{b_1})\right)=B_{\frac{1}{2\sqrt{2}}}(c_1^{b_1})}$. Note that in this case $$|R_b(x)-1|=|R_b^2(x)|=\frac{1}{2\sqrt{2}},$$for all $x\in B_\frac{1}{2\sqrt{2}}(c_1^b)$.

Our goal is find a parameter $b\in\mathbb{C}_K$ such that $R_b$ contains a wandering component in its Fatou set, in fact, we will prove that $b\in D_\frac{1}{4\sqrt{2}}(b_1)$. Therefore, we start our analysis studying the quadratic rational functions $R_\beta$ with $\beta\in D_\frac{1}{4\sqrt{2}}(b_1)$.

\begin{lemma}
\label{LCrit}
Let $R_{b_1}$ the quadratic rational function given in \eqref{Ecb1}. Then $$B_{\frac{1}{2\sqrt{2}}}(c_1^\beta)=B_{\frac{1}{2\sqrt{2}}}(c_1^{b_1})\text{ and }R^3_\beta\left(B_{\frac{1}{2\sqrt{2}}}(c_1^\beta)\right)=B_{\frac{1}{2\sqrt{2}}}(c_1^\beta)$$for all $\beta\in D_{\frac{1}{4\sqrt{2}}}(b_1)$. 
\end{lemma}
\begin{proof}
For each $\beta\in D_{\frac{1}{4\sqrt{2}}}(b_1)$, let us consider the function $f_\beta(z)=\beta z^2-\frac{2z}{a}+\frac{1}{a}$. Note that the zeros of $f_\beta$ are the critical points of $R_\beta$.

Let $x,y\in\mathbb{C}_K$ be points such that $|x|=|y|=\frac{1}{\sqrt{2}}$ and $|x-y|<\frac{1}{2\sqrt{2}}$. Then $|x+y|=\frac{1}{2\sqrt{2}}$ and hence $$|f_\beta(x)-f_\beta(y)|=\left|\beta x^2-\frac{2x}{a}-\beta y^2+\frac{2y}{a}\right|=|x-y|\left|\beta(x+y)-\frac{2}{a}\right|=\frac{|x-y|}{2\sqrt{2}}.$$Therefore, if $|x|=\frac{1}{\sqrt{2}}$ then 
\begin{equation}
\label{Ecf}
    f_\beta(D_{\frac{1}{2\sqrt{2}}}(x))=D_{\frac{1}{8}}(f_\beta(x)).
\end{equation} In particular $f_\beta\left(B_{\frac{1}{2\sqrt{2}}}(c_1^\beta)\right)=B_{\frac{1}{8}}(0).$ The last is true for all $\beta\in D_{\frac{1}{4\sqrt{2}}}(b_1)$. 

Let $\beta\in D_{\frac{1}{4\sqrt{2}}}(b_1)$ be a point, then $\beta=b_1+\delta$ with $|\delta|<\frac{1}{4\sqrt{2}}$. On the one hand, we have that $f_\beta(B_{\frac{1}{2\sqrt{2}}}(c_1^\beta))=B_{\frac{1}{8}}(0)$. And, on the other hand, we have that 
\begin{align*}
|f_\beta(c_1^{b_1})|&=|\beta(c_1^{b_1})^2-\frac{2c_1^{b_1}}{a}+\frac{1}{a}|=|(b_1+\delta)(c_1^{b_1})^2-\frac{2c_1^{b_1}}{a}+\frac{1}{a}|\\
    &=|b_1(c_1^{b_1})^2-\frac{2c_1^{b_1}}{a}+\frac{1}{a}+\delta(c_1^{b_1})^2|=|f_{b_1}(c_1^{b_1})+\delta(c_1^{b_1})^2|\\
    &=|\delta(c_1^{b_1})^2|<\frac{1}{4\sqrt{2}}\cdot\frac{1}{2}=\frac{1}{8\sqrt{2}}<\frac{1}{8}.
\end{align*}
This implies that $f_\beta(c_1^{b_1})\in B_\frac{1}{8}(0)$, therefore we have that $f_\beta(B_\frac{1}{2\sqrt{2}}(c_1^{b_1}))=B_\frac{1}{8}(f_\beta(c_1^{b_1}))=B_\frac{1}{8}(0)$, and hence there exists a point $x\in B_{\frac{1}{2\sqrt{2}}}(c_1^{b_1})$ such that $f_\beta(x)=0$. This point must be a critical point of $R_\beta(z)$, then, without loss generality we can assume that $x=c_1^\beta$. Hence $B_{\frac{1}{2\sqrt{2}}}(c_1^\beta)=B_{\frac{1}{2\sqrt{2}}}(c_1^{b_1})$.

Now we will prove that $R_\beta^3(B_\frac{1}{2\sqrt{2}}(c_1^\beta))=B_\frac{1}{2\sqrt{2}}(c_1^\beta)$ for all $\beta\in D_\frac{1}{4\sqrt{2}}(b_1)$.

Let $\beta\in D_\frac{1}{4\sqrt{2}}(b_1)$ be a point.

{\bf Step 1.} We prove that $R_\beta(B_\frac{1}{2\sqrt{2}}(c_1^\beta))=R_{b_1}(B_\frac{1}{2\sqrt{2}}(c_1^{b_1}))$.

Let $x\in B_\frac{1}{2\sqrt{2}}(c_1^{b_1})$ be a point. Using 4. in Lemma \ref{LR}, if $b\in D_\frac{1}{4\sqrt{2}}(b_1)$ then 
\begin{align*}
|R_{b}(c_1^{b})-R_{b}(x)|&=|a||c_1^{b}-x||b(c_1^{b}-c_1^{b})(c_1^{b}-c_2^{b})+(x-c_1^{b})(bc_1^{b}-\frac{1}{a})|\\
          &=\sqrt{|a|}|c_1^{b}-x|^2\\
          &\leq\sqrt{2}\cdot\frac{1}{8}=\frac{1}{4\sqrt{2}}.
\end{align*} 
This implies that $R_{b}(B_\frac{1}{2\sqrt{2}}(c_1^{b}))=B_\frac{1}{4\sqrt{2}}(R_{b}(c_1^{b}))$, in particular $R_{b_1}(B_\frac{1}{2\sqrt{2}}(c_1^{b_1}))=B_\frac{1}{4\sqrt{2}}(R_{b_1}(c_1^{b_1}))$ and \\ $R_{\beta}(B_\frac{1}{2\sqrt{2}}(c_1^{\beta}))=B_\frac{1}{4\sqrt{2}}(R_{\beta}(c_1^{\beta}))$. 

Since $B_\frac{1}{2\sqrt{2}}(c_1^\beta)=B_\frac{1}{2\sqrt{2}}(c_1^{b_1})$ for all $\beta\in D_{\frac{1}{4\sqrt{2}}}(b_1)$, then $R_\beta(B_\frac{1}{2\sqrt{2}}(c_1^{b_1}))=B_\frac{1}{4\sqrt{2}}(R_\beta(c_1^{b_1}))$.

By 1. in Lemma \ref{RemarkRb} for $z=c_1^{b_1}$, we have that $$|R_{b_1}(c_1^{b_1})-R_\beta(c_1^{b_1})|=|\beta-{b_1}|<\frac{1}{4\sqrt{2}}.$$Therefore $R_{b_1}(c_1^{b_1})\in B_\frac{1}{4\sqrt{2}}(R_\beta(c_1^{b_1}))= R_\beta(B_\frac{1}{2\sqrt{2}}(c_1^{b_1}))= R_\beta(B_\frac{1}{2\sqrt{2}}(c_1^\beta))$, and hence $$R_{b_1}(B_\frac{1}{2\sqrt{2}}(c_1^{b_1}))=R_\beta(B_\frac{1}{2\sqrt{2}}(c_1^\beta)).$$

{\bf Step 2.} We prove that $R_\beta^2(B_\frac{1}{2\sqrt{2}}(c_1^\beta))=R_{b_1}^2(B_\frac{1}{2\sqrt{2}}(c_1^{b_1}))$.

By 1. in Lemma \ref{LR}, since $|R_{\beta}(c_1^{\beta})-1|<1$ for all $\beta\in D_\frac{1}{4\sqrt{2}}(b)$, then $R_\beta(B_\frac{1}{4\sqrt{2}}(R_\beta(c_1^\beta)))=B_\frac{1}{4\sqrt{2}}(R_\beta^2(c_1^\beta))$ for all $\beta\in D_\frac{1}{4\sqrt{2}}(b)$. Therefore, using 2. in Lemma \ref{RemarkRb}, with $z=R_{b_1}(c_1^{b_1})$ (since $R_{b_1}(c_1^{b_1})\in D_\frac{1}{\sqrt{2}}(1)$), we have $$|R_{b_1}(R_{b_1}(c_1^{b_1}))-R_\beta(R_{b_1}(c_1^{b_1}))|=|R_{b_1}(c_1^{b_1})-1||\beta-{b_1}|<\frac{1}{4\sqrt{2}}.$$

Therefore, $R_{b_1}(B_\frac{1}{4\sqrt{2}}(R_{b_1}(c_1^{b_1})))=R_\beta(B_\frac{1}{4\sqrt{2}}R_{b_1}((c_1^{b_1})))=R_\beta(B_\frac{1}{4\sqrt{2}}(R_\beta(c_1^\beta)))$, that is $$R^2_{b_1}(B_\frac{1}{2\sqrt{2}}(c_1^{b_1}))=R^2_\beta(B_\frac{1}{2\sqrt{2}}(c_1^\beta))=B_\frac{1}{4\sqrt{2}}(R_\beta^2(c_1^\beta)).$$

{\bf Step 3.} We prove that $R_\beta^3(B_\frac{1}{2\sqrt{2}}(c_1^\beta))=B_\frac{1}{2\sqrt{2}}(c_1^\beta)$.

Since $|R_\beta^2(c_1^\beta)|=\frac{1}{2\sqrt{2}}<\frac{1}{2}$, then by 2. in Lemma \ref{LR} $R_\beta(B_\frac{1}{4\sqrt{2}}(R_\beta^2(c_1^\beta)))=B_\frac{1}{2\sqrt{2}}(R_\beta^3(c_1^\beta))$ for all $\beta\in D_\frac{1}{4\sqrt{2}}(b_1)$. Therefore, by 3. in Lemma \ref{RemarkRb} with $z=R^2_{b_1}(c_1^{b_1})$, we have $$|R_{b_1}(R_{b_1}^2(c_1^{b_1}))-R_\beta(R_{b_1}^2(c_1^{b_1}))|=2^2|R_{b_1}^2(c_1^{b_1})|^2|\beta-b|=\frac{|\beta-b|}{2}<\frac{1}{4\sqrt{2}},$$and $$R_\beta^3(B_\frac{1}{2\sqrt{2}}(c_1^\beta))=R_\beta(R_\beta^2(B_\frac{1}{2\sqrt{2}}(c_1^\beta)))=R_\beta(R_{b_1}^2(B_\frac{1}{2\sqrt{2}}(c_1^{b_1})))=R_{b_1}(R_{b_1}^2(B_\frac{1}{2\sqrt{2}}(c_1^{b_1})))=R_{b_1}^3(B_\frac{1}{2\sqrt{2}}(c_1^{b_1})).$$Hence, $$R_\beta^3(B_\frac{1}{2\sqrt{2}}(c_1^\beta))=R_{b_1}^3(B_\frac{1}{2\sqrt{2}}(c_1^{b_1}))=B_\frac{1}{2\sqrt{2}}(c_1^{b_1})=B_\frac{1}{2\sqrt{2}}(c_1^\beta).$$
\end{proof} 

\begin{remark}
Lemma \ref{LCrit} imply that $|R_\beta(c_1^\beta)-1|=\frac{1}{2\sqrt{2}}$ and hence $|w_1^\beta-c_1^\beta|=\frac{1}{2}$, for all $\beta\in D_\frac{1}{4\sqrt{2}}(c_1^{b_1})$.
\end{remark}

The following lemma states that the preimages of $1$ under $R_\beta$, $w_1^\beta$ and $w_2^\beta$, are always close, for all $\beta\in D_\frac{1}{4\sqrt{2}}(b_1)$. This implies that $|c_i^{\beta_1}-w_j^{\beta_2}|=\frac{1}{2}$ for $i,j=1,2$ and for all $\beta_1,\beta_2\in D_\frac{1}{4\sqrt{2}}(b_1)$.

\begin{lemma}
\label{Lw1w2}
Let $\beta\in D_\frac{1}{4\sqrt{2}}(b_1)$. Then $|w_1^\beta-w_1^{b_1}|\leq\frac{1}{2\sqrt{2}}$.
\end{lemma}
\begin{proof}
Let us prove this lemma by contradiction. Let us assume that $|w_1^\beta-w_1^{b_1}|>\frac{1}{2\sqrt{2}}$.

Note that $$0=|R_\beta(w_1^\beta)-R_{b_1}(w_1^{b_1})|=|R_\beta(w_1^\beta)-R_\beta(w_1^{b_1})+R_\beta(w_1^{b_1})-R_{b_1}(w_1^{b_1})|,$$therefore $|R_\beta(w_1^\beta)-R_\beta(w_1^{b_1})|=|R_\beta(w_1^{b_1})-R_{b_1}(w_1^{b_1})|$, and moreover, by 1. in Lemma \ref{RemarkRb}
\begin{align*}
|\beta-b_1|&=|R_\beta(w_1^{b_1})-R_{b_1}(w_1^{b_1})|=|R_\beta(w_1^\beta)-R_\beta(w_1^{b_1})|\\
     &=2|w_1^\beta-w_2^{b_1}||\beta(w_1^\beta-c_1^\beta)(w_1^\beta-c_2^\beta)+(w_1^{b_1}-w_1^\beta)(\beta w_1^\beta-\frac{1}{a})|\\
     &=\sqrt{2}|w_1^\beta-w_1^{b_1}|^2>\sqrt{2}\frac{1}{8}=\frac{1}{4\sqrt{2}}.
\end{align*}
but this is a contradiction because by hypothesis $|\beta-b_1|<\frac{1}{4\sqrt{2}}$. Hence $|w_1^\beta-w_1^{b_1}|\leq\frac{1}{2\sqrt{2}}$.
\end{proof}

Lemma \ref{LCrit} and Lemma \ref{Lw1w2} imply that $$c_1^\beta,c_2^\beta\in B_\frac{1}{2\sqrt{2}}(c_1^{b_1})\hspace{1cm}\text{and}\hspace{1cm}w_1^\beta,w_2^\beta\in B_\frac{1}{2\sqrt{2}}(w_1^{b_1}),$$ for all $\beta\in D_\frac{1}{4\sqrt{2}}(b_1)$.
%Let us denote $\mathfrak{D}:=D_\frac{1}{2\sqrt{2}}(c_1^b)$. By Lemma \ref{LCrit} $\mathcal{D}=D_\frac{1}{2\sqrt{2}}(c_1^\beta)$ for all $\beta\in D_\frac{1}{4\sqrt{2}}(b)$.

Now, we study the dynamics of $R_\beta$ with $\beta\in D_\frac{1}{4\sqrt{2}}(b_1)$, in order to do this we define some notations. 

Let us define the numbers $\rho_1=\frac{1}{2}$ and $$\rho_i=\frac{1}{2\sqrt[2^2]{2}\cdots\sqrt[2^i]{2}}$$for $i\geq2$. 

The following lemma give us some properties of the sequence $\{\rho_i\}_{i\geq1}$.

\begin{lemma}
\label{Lrhozero}
\begin{enumerate}
\hfill
    \item $\rho_i=2\sqrt{2}\rho_{i+1}^2$, for all $i\geq1$.
    \item $\rho_i>\frac{1}{2\sqrt{2}}$, for all $i\geq1$.
    \item $\rho_1>\rho_2>\cdots$.
    \item $\lim_{N\to\infty}2^{2N}\rho_N^2\rho_{N-1}^2\cdots\rho_1^2=0.$
\end{enumerate}
\end{lemma}
\begin{proof}
\hfill
\begin{enumerate}
\item By definition of $\rho_{i+1}$
\begin{align*}
2\sqrt{2}\rho_{i+1}^2&=2\sqrt{2}\left(\frac{1}{2\sqrt[2^2]{2}\sqrt[2^3]{2}\cdots\sqrt[2^{i+1}]{2}}\right)^2\\
  &=2\sqrt{2}\cdot\frac{1}{4\sqrt{2}\sqrt[2^2]{2}\cdots\sqrt[2^i]{2}}\\
  &=\frac{1}{2\sqrt[2^2]{2}\cdots\sqrt[2^i]{2}}=\rho_i.
\end{align*}
And, if $i=1$ then $$2\sqrt{2}\rho_2^2=2\sqrt{2}\left(\frac{1}{2\sqrt[2^2]{2}}\right)^2=2\sqrt{2}\cdot\frac{1}{4\sqrt{2}}=\frac{1}{2}=\rho_1.$$
\item Note that, $$\rho_i=\frac{1}{2\sqrt[2^2]{2}\sqrt[2^3]{2}\cdots\sqrt[2^i]{2}}=\frac{1}{2\cdot2^{1/4}2^{1/8}\cdots2^{1/2^i}}=\frac{1}{2\cdot2^{1/4+1/8+\cdots1/2^i}}>\frac{1}{2\cdot2^\frac{1}{2}}=\frac{1}{2\sqrt{2}}.$$
\item Using 1. and 2. in this lemma we have $$\rho_i=2\sqrt{2}\rho_{i+1}^2=(2\sqrt{2}\rho_{i+1})\rho_{i+1}>\left(2\sqrt{2}\cdot\frac{1}{2\sqrt{2}}\right)\rho_{i+1}=\rho_{i+1}.$$
\item Note that $2^2\rho_i^2=2(2\rho_i^2)=2\frac{\rho_{i-1}}{\sqrt{2}}=\sqrt{2}\rho_{i-1}\leq\frac{\sqrt{2}}{2}=\frac{1}{\sqrt{2}}$ for all $i>1$, and $2^2\rho_1^2=1$. Therefore, $$2^{2N}\rho_N^2\rho_{N-1}^2\cdots\rho_1^2=(2^2\rho_N^2)(2^2\rho_{N-1}^2)\cdots(2^2\rho_1^2)\leq\frac{1}{(\sqrt{2})^{N-1}},$$hence $\lim_{N\to\infty}2^{2N}\rho_N^2\rho_{N-1}^2\cdots\rho_1^2=0$.
\end{enumerate}

\end{proof}

Since $\rho_i>\frac{1}{2\sqrt{2}}$ and by Lemma \ref{LCrit} we have that $\mathbb{S}_{\rho_i}(c_1^{b_1})=\mathbb{S}_{\rho_i}(c_1^\beta)$ for all $i\geq1$, and for all $\beta \in D_\frac{1}{4\sqrt{2}}(b_1)$. Therefore, we define the following notation.

\begin{equation}
\label{Not1}
\Psi_i=\mathbb{S}_{\rho_i}(c_1^{b_1})\hspace{1cm}\text{ and }\hspace{1cm}\Gamma_i=\mathbb{S}_\frac{1}{2^i\sqrt{2}}(0)
\end{equation}

\begin{lemma}
\label{Lrhoi}
Let $\beta\in D_\frac{1}{4\sqrt{2}}(b_1)$ be a parameter. Let $i\geq1$,
\begin{enumerate}
\item if $x\in\Psi_{i+1}$, then $R^3_\beta(x)\in\Psi_i$.
\item if $x\in\Gamma_{i+1}$ then $R_\beta(x)\in\Gamma_i$.
\end{enumerate}
\end{lemma}
\begin{proof}
\hfill
\begin{enumerate}
\item Since $x\in\Psi_{i+1}$ then $|x-c_1^{b_1}|=\rho_{i+1}<\frac{1}{2}$. By 4. in Lemma \ref{LR} we have that $$|R(x)-R(c_1^{b_1})|=\sqrt{2}|x-c_1^{b_1}|^2=\sqrt{2}\rho_{i+1}^2.$$By 1. in Lemma \ref{LR} $$|R^2(x)-R^2(c_1^{b_1})|=|R(x)-R(c_1^{b_1})|=\sqrt{2}\rho_{i+1}^2.$$By 2. in Lemma \ref{LR} we have $$|R^3(x)-R^3(c_1^{b_1})|=2|R^2(x)-R^2(c_1^{b_1})|=2\sqrt{2}\rho_{i+1}^2=\rho_i.$$Finally, since $|R^3(c_1^{b_1})-c_1^{b_1}|\leq\frac{1}{2\sqrt{2}}<\rho_i$ then $$|R^3(x)-c_1^{b_1}|=|R^3(x)-R^3(c_1^{b_1})+R^3(c_1^{b_1})-c_1^{b_1}|=\rho_i.$$Therefore, if $x\in\Psi_{i+1}$, then $R^3_\beta(x)\in\Psi_i$.
\item Recall that if $|z|<\frac{1}{2}$ then $|R_\beta(x)|=2|x|$. Therefore, if $x\in\Gamma_{i+1}$ then $|x|=\frac{1}{2^{i+1}\sqrt{2}}$ and $$|R_\beta(x)|=2|x|=\frac{2}{2^{i+1}\sqrt{2}}=\frac{1}{2^i\sqrt{2}}.$$Hence $R_\beta(x)\in\Gamma_i$.
\end{enumerate}
\end{proof}

As a consequence of Lemma \ref{LR} we have the following lemma, which give us the dynamics of $R_\beta$ over important subsets of $\mathbb{P}(\mathbb{C}_K)$.

\begin{lemma}
\label{LRb}
Let $\beta\in D_\frac{1}{4\sqrt{2}}(b)$. Then $R_\beta$ has the following dynamics.
\begin{enumerate}
\item Assume that $D_r(y)\subset\Psi_k$ with $r\leq\sqrt{2}\rho_k^2$, then $R_\beta(D_r(y))=D_{2r\rho_k^2}(R_\beta(y))$. Otherwise, if $r>\sqrt{2}\rho_k^2$, then $R_\beta(D_r(y))=D_{\sqrt{2}r^2}(R_\beta(y))$.
\item If $D_r(y)\subset D_\frac{1}{\sqrt{2}}(1)$ then $R_\beta(D_r(y))=D_r(R_\beta(y))$.
\item If $D_r(y)\subset D_\frac{1}{2}(0)$ then $R_\beta(D_r(y))=D_{2r}(R_\beta(y))$.
\end{enumerate}
\end{lemma}
\begin{proof}
It follows directly from Lemma \ref{LR}.
\end{proof}

\begin{lemma}
Let $\beta\in D_\frac{1}{4\sqrt{2}}(b_1)$, $i>1$ and $D_r(x)\subset\Psi_i$ with $r<\sqrt{2}\rho_i^2$, then $$R_\beta^3(D_r(x))=D_\delta(R_\beta^3(x))\subset\Psi_{i-1}\text{ with } \delta=2^2\rho_i^2r<\sqrt{2}\rho_{i-1}^2.$$
\end{lemma}
\begin{proof}
Let $y\in D_r(x)$, then $|x-y|<r<\sqrt{2}\rho_i^2$. Since $D_r(x)\subset\Psi_i$ then $|x-c_1^\beta|=|x-c_2^\beta|=\rho_i$ and hence $|x-c_1^\beta||x-c_2^\beta|=\rho_i^2>\frac{r}{\sqrt{2}}>\frac{|x-y|}{\sqrt{2}}$. Therefore 
\begin{align*}
|R_\beta(x)-R_\beta(y)|&=|a||x-y||\beta(x-c_1^\beta)(x-c_2^\beta)+(y-x)(\beta x-\frac{1}{a})|\\
    &=2|x-y|\rho_i^2.
\end{align*}
This implies that $R_\beta(D_r(x))=D_{2\rho_i^2r}(R_\beta(x))$.

Since $R_\beta$ is an isomorphism in $D_1(1)$ then $R^2_\beta(D_r(x))=D_{2\rho_i^2r}(R_\beta^2(x))$. Finally, if $D_s(y)\subset D_\frac{1}{|a|}(0)$ then $R(D_s(y))=D_{2s}(R(y))$, therefore, $R^3(D_r(x))=D_\delta(R^3(x))$ with $\delta=2^2\rho_i^2r$.

Note that 
\begin{align*}
\delta&=2^2\rho_i^2r=4\frac{\rho_{i-1}}{2\sqrt{2}}r=\sqrt{2}\rho_{i-1}r\\
      &<\sqrt{2}\rho_{i-1}(\sqrt{2}\rho_i^2)=\sqrt{2}\rho_{i-1}(\frac{\rho_{i-1}}{2})\\
      &=\frac{\rho_{i-1}^2}{\sqrt{2}}<\sqrt{2}\rho_{i-1}^2
\end{align*}

This and 1. in Lemma \ref{Lrhoi} imply that $R_\beta^3(D_r(x))=D_\delta(R_\beta^3(x))\subset\Psi_{i-1}\text{ with } \delta=2^2\rho_i^2r<\sqrt{2}\rho_{i-1}^2.$
\end{proof}

\subsection{Construction of the sequence}
% Let $r<\frac{\sqrt{2}\rho_{m_1}^2}{2^{n_1}}$.
In this subsection we construct a sequence of integer numbers, this sequence will describe the orbit of an special point $x$ which will be in a wandering component (of some quadratic rational function). Therefore, the sequence must be chosen appropriately to avoid that the point $x$ being in a preperiodic component of the Fatou set.   

Let $n_1,n_2,\ldots$ be a sequence of integer numbers such that $n_1>1$ and $n_{k+1}\geq n_k+2$ for all $k\geq1$. By induction, for $k=1,2,\ldots$ we choose the integers $m_k$ such that $$2^{2m_k}\rho_{m_k}^2\cdots\rho_1^2<\frac{1}{2^{n_{k+1}+1}}.$$

Roughly speaking, the numbers $n_k$ represent the number of consecutive times the images of the point $x$ is close to $0$ and the numbers $m_k$ represent the number of times the images of $x$ is close to the critical points, before it gets close to $0$ again. 

Given the sequence $n_1,m_1,n_2,m_2,\ldots$, let us define the numbers

\begin{equation}
\label{I}
I_0:=0\hspace{5mm}\text{ and }\hspace{5mm} I_k=\sum_{j=1}^k(n_j+3(m_j-1)+2).
\end{equation}

\begin{definition}
We say that $x$ has itinerary $n_1m_1n_2m_2\cdots n_km_k$ under $R_\beta$ if
\begin{itemize}
\item $x\in D_\frac{1}{\sqrt{2}}(1)$,
\item $R_\beta^{I_j+1}(x)\in\Gamma_{n_{j+1}}$ for $j=0,\ldots,k-1$,
\item $R_\beta^{I_j+n_{j+1}+1}(x)\in\Psi_{m_{j+1}}$ for $j=0,\ldots,k-1$,
\item $R_\beta^{I_k}(x)\in D_\frac{1}{\sqrt{2}}(1)$.
\end{itemize}
We say that $x$ has an strict itinerary under $R_\beta$, if moreover $R_\beta^{I_k}(x)=1$. 
\end{definition}

\subsection{The point $x$ with strict itinerary $n_1m_1$ under $R_{b_1}$}

We have started with a special quadratic rational function $R_{b_1}(z)$, such that $R_{b_1}^3(B_\frac{1}{2\sqrt{2}}(c_1^{b_1}))=B_\frac{1}{2\sqrt{2}}(c_1^{b_1})$. In this subsection we find a point $x\in D_\frac{1}{\sqrt{2}}(1)$ such that it has strict itinerary $n_1m_1$ under $R_{b_1}(z)$.

\begin{lemma}
Let $R_{b_1}(z)$ be a quadratic rational function as in \eqref{Ecb1} and $n_1$ and $m_1$ integer numbers as in previous subsection. Then, there exists a point $x$ with strict itinerary $n_1m_1$ under $R_{b_1}$
\end{lemma}
\begin{proof}
We prove this lemma in three steps.

{\bf Step 1.} Let us consider the sphere $\mathbb{S}_\frac{1}{2^{n_1}\sqrt{2}}(1)$. Note that $R_{b_1}(\mathbb{S}_\frac{1}{2^{n_1}\sqrt{2}}(1))=\Gamma_{n_1}$ and hence $$R_{b_1}^{n_1+1}(\mathbb{S}_\frac{1}{2^{n_1}\sqrt{2}}(1))=\mathbb{S}_\frac{1}{\sqrt{2}}(0).$$Therefore, there exists $y\in\mathbb{S}_\frac{1}{2^{n_1}\sqrt{2}}(1)$ such that $R_{b_1}^{n_1+1}(D_\frac{1}{2^{n_1}\sqrt{2}}(y))=D_\frac{1}{\sqrt{2}}(c_1^{b_1})$. Moreover, if $R_{b_1}^{n_1+1}(y)=c_1^{b_1}$ then $$R_{b_1}^{n_1+1}(\mathbb{S}_\frac{\rho_{m_1}}{2^{n_1}}(y))=\mathbb{S}_{\rho_{m_1}}(c_1^{b_1})=\Psi_{m_1}.$$

{\bf Step 2.} Let us consider the sphere $\Psi_{m_1}$. By Lemma \ref{Lrhoi} we have that $$R_{b_1}^{3(m_1-1)}(\Psi_{m_1})=\Psi_1=\mathbb{S}_\frac{1}{2}(c_1^{b_1}).$$Therefore, there exists a point $p\in\Psi_{m_1}$ such that $R_{b_1}^{3(m_1-1)}(p)=w_1^{b_1}$ and hence $R_{b_1}^{3(m_1-1)+1}(p)=1$.

{\bf Step 3.} Let $x\in\mathbb{S}_\frac{\rho_{m_1}}{2^{n_1}}(y)$ such that $R_{b_1}^{n_1+1}(x)=p$. Hence,
\begin{itemize}
\item Since $x\in\mathbb{S}_\frac{\rho_{m_1}}{2^{n_1}}(y)\subset\mathbb{S}_\frac{1}{2^{n_1}\sqrt{2}}(1)$ then $x\in D_\frac{1}{\sqrt{2}}(1)$.
\item Since $x\in\mathbb{S}_\frac{1}{2^{n_1}\sqrt{2}}(1)$ then $R_{b_1}^{I_0+1}(x)\in\Gamma_{n_1}$.
\item Since $R_{b_1}^{n_1+1}(x)=p$ then $R_{b_1}^{I_0+n_1+1}(x)\in\Psi_{m_1}$.
\item Since $R_{b_1}^{3(m_1-1)+1}(p)=1$ then $$R_{b_1}^{I_1}(x)=R_{b_1}^{I_0+n_1+3(m_1-1)+2}(x)=R_{b_1}^{3(m_1-1)+1}(R_{b_1}^{n_1+1}(x))=R_{b_1}^{3(m_1-1)+1}(p)=1.$$ 
\end{itemize}
Therefore, $x$ has itinerary $n_1m_1$ under $R_{b_1}$.
\end{proof}

Now, we vary the parameter $b_1$ so that the point $x$ has bigger itineraries, that is, we find a parameter $b_2\in D_\frac{1}{4\sqrt{2}}(b_1)$ such that $x$ has itinerary $n_1m_1n_2m_2$ under $R_{b_2}(z)$, then we find a parameter $b_3\in D_\frac{1}{4\sqrt{2}}(b_1)$ such that $x$ has itinerary $n_1m_1n_2m_2n_3m_3$ under $R_{b_3}(z)$, and so on. Our goal is to find a parameter $b\in D_\frac{1}{4\sqrt{2}}(b_1)$ such that the point $x$ has itinerary $n_1m_1n_2m_2\ldots$ under $R_b(z)$.

Before to find the parameters $b_2$, $b_3,\ldots$ we study, in the next subsection, the dynamics of the rational function $R_\beta(z)$ for different $\beta\in D_\frac{1}{4\sqrt{2}}(b_1)$.

\subsection{The family $\Phi_x^N$}

In this subsection we compare the points $R^n_\beta(x)$ for different $\beta\in D_\frac{1}{4\sqrt{2}}(b_1)$ and different values $n\in\mathbb{N}$. This analysis will help us to find the next parameters, $b_2$, $b_3$, etcetera. In order to compare the values $R_\beta(x)$ we give the following definition.

\begin{definition}
\label{Phi}
Given $N\in\mathbb{N}$, let us define the rational function 
\begin{align*}
\Phi_x^N:\mathbb{P}(\mathbb{C}_K)&\to\mathbb{P}(\mathbb{C}_K)\\
              \beta&\mapsto R_\beta^N(x)
\end{align*}
\end{definition}

Note that if $x$ has strict itinerary $n_1m_1\cdots n_km_k$ under $R_\beta$ then $$\Phi_x^{I_k}(\beta)=R_\beta^{I_k}(x)=1.$$

\begin{remark}
\label{RPhi1}
Let $x\in D_\frac{1}{\sqrt{2}}(1)$ and $\beta\in D_\tau(b_1)$ with $\tau\leq\frac{1}{4\sqrt{2}}$. Then, by 2. in Lemma \ref{RemarkRb} $$|\Phi_x^1(\beta)-\Phi_x^1(b_1)|=|R_\beta(x)-R_{b_1}(x)|=|x-1||\beta-b|,$$and hence $\Phi_x^1(D_\tau(b_1))=D_{|x-1|\tau}(R_1(x))$. 
\end{remark}

Moreover, we have the following lemma.

\begin{lemma}
\label{LPhi}
Let $\Phi_x^N$ be the rational function given in Definition \ref{Phi}. 
\begin{enumerate}
\item Assume that $\Phi_x^N(D_\tau(b))=D_\delta(R_b^N(x))\subset\Psi_k$ and let $\mu_1=\max\{\tau,2\delta\rho_k^2,\sqrt{2}\delta^2\}$. Then $$\Phi_x^{N+1}(D_\tau(b))=D_{\mu_1}(R_b^{N+1}(x)).$$
\item Assume that $\Phi_x^N(D_\tau(b))=D_\delta(R_b^N(x))\subset D_\frac{1}{\sqrt{2}}(1)$ and let $\mu_2=\max\{|R_b^N(x)-1|\tau,\delta\}$. Then $$\Phi_x^{N+1}(D_\tau(b))=D_{\mu_2}(R_b^{N+1}(x)).$$
\item Assume that $\Phi_x^N(D_\tau(b))=D_\delta(R_b^N(x))\subset D_\frac{1}{2}(0)$ and let $\mu_3=\max\{4|R_b^N(x)|^2\tau,2\delta\}$. Then $$\Phi_x^{N+1}(D_\tau(b))=D_{\mu_3}(R_b^{N+1}(x)).$$
\end{enumerate}
\end{lemma}
\begin{proof}
In order to prove this lemma we use Lemma \ref{RemarkRb} and the fact
\begin{align*}
|\Phi_x^{N+1}(b)-\Phi_x^{N+1}(\beta)|&=|R_b^{N+1}(x)-R_\beta^{N+1}(x)|\\
  &=|R_b(R_b^N(x))-R_\beta(R_\beta^N(x))|\\
  &=|R_b(R_b^N(x))-R_\beta(R_b^N(x))+R_\beta(R_b^N(x))-R_\beta(R_\beta^N(x))|
\end{align*}
\begin{enumerate}
\item Since $\Phi_x^N(D_\tau(b))=D_\delta(R_b^N(x))\subset\Psi_k$ then $$\delta=\sup_{\beta\in D_\tau(b)}\{|\Phi_x^N(b)-\Phi_x^N(\beta)|\}=\sup_{\beta\in D_\tau(b)}\{|R_b^N(x)-R_\beta^N(x)|\}.$$Note that, by 1. in Lemma \ref{RemarkRb}, $|R_b(R_b^N(x))-R_\beta(R_b^N(x))|=|b-\beta|$ and 
\begin{align*}
|R_\beta(R_b^N(x))-&R_\beta(R_\beta^N(x))|\\
&=2|R_b^N(x)-R_\beta^N(x)|\left|\beta(R_b^N(x)-c_1^\beta)(R_b^N(x)-c_2^\beta)-(R_b^N(x)-R_\beta^N(x))\left(\beta R_b^N(x)-\frac{1}{a}\right)\right|
\end{align*}
where $$|\beta(R_b^N(x)-c_1^\beta)(R_b^N(x)-c_2^\beta)|=\rho_k^2$$ and $$|(R_b^N(x)-R_\beta^N(x))(\beta R_b^N(x)-\frac{1}{a})|=\frac{|(R_b^N(x)-R_\beta^N(x))|}{\sqrt{2}}.$$

Therefore $$|\Phi_x^{N+1}(b)-\Phi_x^{N+1}(\beta)|\leq\max\{|b-\beta|,2|R_b^N(x)-R_\beta^N(x)|\rho_k^2,\sqrt{2}|R_b^N(x)-R_\beta^N(x)|^2\}.$$Since $\tau=\sup_{\beta\in D_\tau(b)}\{|b-\beta|\}$ and $\delta=\sup_{\beta\in D_\tau(b)}\{|R_b^N(x)-R_\beta^N(x)|\}$, then $\diam(\Phi_x^{N+1}(D_\tau(b)))=\mu_1$ with $\mu_1=\max\{\tau,2\delta\rho_k^2,\sqrt{2}\delta^2\}$.

\item By 2. in Lemma \ref{RemarkRb} and since $\Phi_x^N(D_\tau(b))=D_\delta(R_b^N(x))\subset D_\frac{1}{\sqrt{2}}(1)$ then $$|R_b(R_b^N(x))-R_\beta(R_b^N(x))|=|R_b^N(x)-1||b-\beta|$$ and $|R_\beta(R_b^N(x))-R_\beta(R_\beta^N(x))|=|R_b^N(x)-R_\beta^N(x)|$. Therefore 
\begin{align*}
|\Phi_x^{N+1}(b)-\Phi_x^{N+1}(\beta)|&=|R_b(R_b^N(x))-R_\beta(R_b^N(x))+R_\beta(R_b^N(x))-R_\beta(R_\beta^N(x))|\\
    &\leq\max\{|R_b^N(x)-1||b-\beta|,|R_b^N(x)-R_\beta^N(x)|\}.
\end{align*}
Since $\tau=\sup_{\beta\in D_\tau(b)}\{|b-\beta|\}$ and $\delta=\sup_{\beta\in D_\tau(b)}\{|R_b^N(x)-R_\beta^N(x)|\}$, then $\diam(\Phi_x^{N+1}(D_\tau(b)))=\mu_2$ with $\mu_2=\max\{|R_b^N(x)-1|\tau,\delta\}$.

\item By 3. in Lemma \ref{RemarkRb} and since $\Phi_x^N(D_\tau(b))=D_\delta(R_b^N(x))\subset D_\frac{1}{2}(0)$ then $$|R_b(R_b^N(x))-R_\beta(R_b^N(x))|=4|R_b^N(x)|^2|b-\beta|$$ and $|R_\beta(R_b^N(x))-R_\beta(R_\beta^N(x))|=2|R_b^N(x)-R_\beta^N(x)|$. Therefore 
\begin{align*}
|\Phi_x^{N+1}(b)-\Phi_x^{N+1}(\beta)|&=|R_b(R_b^N(x))-R_\beta(R_b^N(x))+R_\beta(R_b^N(x))-R_\beta(R_\beta^N(x))|\\
    &\leq\max\{4|R_b^N(x)|^2|b-\beta|,2|R_b^N(x)-R_\beta^N(x)|\}.
\end{align*}
Since $\tau=\sup_{\beta\in D_\tau(b)}\{|b-\beta|\}$ and $\delta=\sup_{\beta\in D_\tau(b)}\{|R_b^N(x)-R_\beta^N(x)|\}$, then $\diam(\Phi_x^{N+1}(D_\tau(b)))=\mu_3$ with $\mu_3=\max\{4|R_b^N(x)|^2\tau,2\delta\}$.
\end{enumerate}
\end{proof}

Now, we define some values ($\tau_k$ for $k\geq1$) which will help us to find the parameters $b_k$. In fact, in the next subsection, we will prove that the parameter $b_{k+1}\in D_{\tau_k}(b_k)$ for all $k\geq1$, that is, $b_2\in D_{\tau_1}(b_1)$, $b_3\in D_{\tau_2}(b_2)$, and so on.

For $k\geq1$, we choose numbers $\tau_k>0$ such that $$\tau_1>\tau_2>\cdots\hspace{5mm}\text{and}\hspace{5mm}\frac{1}{2^{n_{k+1}}\sqrt{2}}<\tau_k<\frac{\sqrt{2}\rho_{m_k}^2}{2^{n_k}}\text{ for all }k\geq1.$$Note that $$\tau_k<\frac{\sqrt{2}\rho_{m_k}^2}{2^{n_k}}=\sqrt{2}\rho_{m_k}\frac{\rho_{m_k}}{2^{n_k}}<\frac{\rho_{m_k}}{2^{n_k}}$$for all $k\geq1$.

The following proposition give us some control over the diameter of the images $\Phi_x^{I_d}(D_{\tau}(b))$.
 
\begin{proposition}
\label{PTau}
Let $b\in D_\frac{1}{4\sqrt{2}}(b_1)$ and let $x\in D_\frac{1}{\sqrt{2}}(1)$ be a point with itinerary $n_1m_1\cdots n_km_k$ under $R_b$. Then $\Phi_x^{I_d}(D_{\tau}(b))=D_{\tau}(R_{b}^{I_d}(x))$ for all $d=1,\ldots,k$ and for all $\tau\leq\tau_k$.
\end{proposition}
\begin{proof}
We will prove this proposition by induction over $d$. 

{\bf Step 1.} Let assume that $d=1$, that is, $x$ has itinerary $n_1m_1$. Let $\tau\leq\tau_k$. Since $x$ has itinerary $n_1m_1$ under $R_b$ then $R_b(x)\in\Gamma_{n_1}$, this implies that $|x-1|=|R_b(x)|=\frac{1}{2^{n_1}\sqrt{2}}$. By Remark \ref{RPhi1} we have that $$\Phi_x^1(D_{\tau}(b))=D_{\frac{\tau}{2^{n_1}\sqrt{2}}}(R_b(x))\subset\Gamma_{n_1}.$$

Note that for $j\in\{1,2,\ldots,n_1\}$ $$4|R_b^j(x)|^2=4\left(\frac{1}{2^{n_1-j+1}\sqrt{2}}\right)^2=\left(\frac{4}{2^{n_1-j+1}\sqrt{2}}\right)\left(\frac{1}{2^{n_1-j+1}\sqrt{2}}\right)<2\left(\frac{1}{2^{n_1-j+1}\sqrt{2}}\right),$$
then by 3. in Lemma \ref{LPhi} we have that for $j\in\{1,2,\ldots,n_1\}$ $$\Phi_x^{j+1}(D_{\tau}(b))=D_{\frac{\tau}{2^{n_1-j}\sqrt{2}}}(R_b^{j+1}(x)).$$In particular, $\Phi_x^{n_1+1}(D_{\tau}(b))=D_{\frac{\tau}{\sqrt{2}}}(R_b^{n_1+1}(x))$.

By hypothesis $\tau\leq\tau_k<\tau_1<\frac{\rho_{m_1}}{2^{n_1}}$ and using the fact $\frac{1}{2\sqrt{2}}<\rho_i\leq\frac{1}{2}$ for all $i\geq1$, we have that $$\tau\leq\tau_1<\frac{\rho_{m_1}}{2^{n_1}}\leq\frac{1}{2^{n_1+1}}<\frac{1}{2}<\sqrt{2}\rho_{m_1}.$$This implies that $\frac{\tau}{\sqrt{2}}<\rho_{m_1}$ and hence $D_{\frac{\tau}{\sqrt{2}}}(R_b^{n_1+1}(x))\subset\Psi_{m_1}$.

By 1. in Lemma \ref{LPhi} with $\delta=\frac{\tau}{\sqrt{2}}$ and $N=n_1+1$ we have that $\tau=\max\{\tau,2\delta\rho_{m_1}^2,\sqrt{2}\delta^2\}$ and hence $$\Phi_x^{n_1+2}(D_{\tau}(b))=D_{\tau}(R_b^{n_1+2}(x))\subset\mathbb{S}_\frac{1}{2\sqrt{2}}(1)\subset D_\frac{1}{\sqrt{2}}(1).$$

By 2. in Lemma \ref{LPhi} with $\delta=\tau$ and $N=n_1+2$ we have that $\tau=\max\{|R_b^N(x)-1|\tau,\tau\}$ and hence$$\Phi_x^{n_1+3}(D_{\tau}(b))=D_{\tau}(R_b^{n_1+3}(x))\subset\mathbb{S}_\frac{1}{2\sqrt{2}}(0)\subset D_\frac{1}{2}(0).$$

By 3. in Lemma \ref{LPhi} with $\delta=\tau$ and $N=n_1+3$ we have that $2\tau=\max\{4|R_b^N(x)|^2\tau,2\tau\}$ and hence $$\Phi_x^{n_1+4}(D_{\tau}(b))=D_{2\tau}(R_b^{n_1+4}(x))\subset\Psi_{m_1-1}.$$

By 1. in Lemma \ref{LPhi} with $\delta=2\tau$ and $N=n_1+4$ we have that $\tau=\max\{\tau,2\delta\rho_{m_1}^2,\sqrt{2}\delta^2\}$ and hence$$\Phi_x^{n_1+5}(D_{\tau}(b))=D_{\tau}(R_b^{n_1+5}(x))\subset D_\frac{1}{\sqrt{2}}(1),$$and so on.

Following the previous analysis we conclude that $$\Phi_x^{I_0+n_1+3(m_1-1)+2}(D_{\tau}(b))=\Phi_x^{I_1}(D_{\tau}(b))=D_{\tau}(R_b^{I_1}(x)).$$

{\bf Step 2.} Now, let us assume that the proposition is true for $1\leq d<k$. We will prove that the proposition is true for $d+1$.

By hypothesis $x$ is a point with itinerary $n_1m_1\cdots n_km_k$ and $\Phi_x^{I_{d_1}}(D_{\tau}(b))=D_{\tau}(R_b^{I_{d_1}}(x))$ for all $d_1=1,\ldots,d$ and for all $\tau\leq\tau_k$. 

Now, let us consider the itinerary $n_1m_1\cdots n_{d+1}m_{d+1}$, then by hypothesis, $$\Phi_x^{I_d}(D_{\tau}(b))=D_{\tau}(R_b^{I_d}(x)).$$   

By 2. in Lemma \ref{LPhi} with $\delta=\tau$ and $N=I_d$ we have that $$\Phi_x^{I_d+1}(D_{\tau}(b))=D_{\tau}(R_b^{I_d+1}(x))\subset\Gamma_{n_{d+1}}.$$

Let $j\in\{1,2,\ldots,n_{d+1}\}$ and $\delta=2^{j-1}\tau$. By 3. in Lemma \ref{LPhi}, if $$\Phi_x^{I_d+j}(D_{\tau}(b))=D_\delta(R_b^{I_d+j}(x))\subset\Gamma_{n_{d+1}-j+1},$$then $$2\delta=\max\{4|R_b^{I_d+j}(x)|^2\tau,2\delta\}$$and hence$$\Phi_x^{I_d+j+1}(D_{\tau}(b))=D_{2^j\tau}(R_b^{I_d+j+1}(x))\subset\Gamma_{n_{d+1}-j}.$$By induction we have that $$\Phi_x^{I_d+n_{d+1}+1}(D_{\tau}(b))=D_{2^{n_{d+1}}\tau}(R_b^{I_d+n_{d+1}+1}(x))\subset\Gamma_0.$$

Since $x$ has itinerary $n_1m_1\cdots n_{d+1}m_{d+1}$ under $R_b$, then $R_b^{I_d+n_{d+1}+1}(x)\in\Psi_{m_{d+1}}$. Moreover, since $\tau\leq\tau_k\leq\tau_{d+1}<\frac{\rho_{m_{d+1}}}{2^{n_{d+1}}}$ then $2^{n_{d+1}}\tau<\rho_{m_{d+1}}$, and therefore $$\Phi_x^{I_d+n_{d+1}+1}(D_{\tau}(b))=D_{2^{n_{d+1}}\tau}(R_b^{I_d+n_{d+1}+1}(x))\subset\Psi_{m_{d+1}}.$$

Note that, for $\delta_0=2^{n_{d+1}}\tau$ and $\delta_j=2^{n_{d+1}}2^{2j}\rho_{m_{d+1}}^2\cdots\rho_{m_{d+1}-j+1}^2\tau$ for $1\leq j< m_{d+1}$ we have the following:
\begin{itemize}
\item since $\tau\leq\tau_{d+1}<\frac{\sqrt{2}\rho_{m_{d+1}}^2}{2^{n_{d+1}}}$ and $2^2\rho_i^2\leq1$ for all $i\geq1$, then $\delta_j\leq2^{n_{d+1}}\tau$ and hence $$\sqrt{2}\delta_j^2=(\sqrt{2}\delta_j)(\delta_j)\leq(\sqrt{2}\delta_j)2^{n_{d+1}}\tau<(\sqrt{2}\delta_j)(\sqrt{2}\rho_{m_{d+1}}^2)\leq2\delta_j\rho_{m_{d+1}-j}^2,$$for all $j=0,1,\ldots,m_{d+1}$.
\item $2\delta_0\rho_{m_{d+1}}^2=2^{n_{d+1}+1}\rho_{m_{d+1}}^2\tau>2^{n_d+3}\cdot\frac{1}{8}\tau=2^{n_d}\tau>\tau$.
\item \begin{align*}
2\delta_{m_{d+1}-1}\rho_1^2&=2(2^{n_{d+1}}2^{2(m_{d+1}-1)}\rho_{m_{d+1}}^2\cdots\rho_2^2\tau)\rho_1^2\\
 &=\frac{1}{2}\cdot2^{n_{d+1}}2^{2(m_{d+1})}\rho_{m_{d+1}}^2\cdots\rho_2^2\rho_1^2\tau<\frac{1}{2}\cdot\frac{2^{n_{d+1}}}{2^{n_{d+1}+1}}\tau\\
 &=\frac{\tau}{4}<\tau
\end{align*}
\item $\delta_{j+1}<\delta_j$ for all $j\geq0$.
\end{itemize}
Therefore, there exists $J\in\{0,\ldots,m_{d+1}-2\}$ such that $$2\delta_{J+1}\rho_{m_{d+1}-J-1}^2<\tau<2\delta_J\rho_{m_{d+1}-J}^2$$

Recall that $$\Phi_x^{I_d+n_{d+1}+1}(D_{\tau}(b))=D_{\delta_0}(R_b^{I_d+n_{d+1}+1}(x))\subset\Psi_{m_{d+1}},$$therefore, by the previous remark we have that $$\Phi_x^{I_d+n_{d+1}+3j+1}(D_{\tau}(b))=D_{\delta_j}(R_b^{I_d+n_{d+1}+3j+1}(x))\subset\Psi_{m_{d+1}-j},$$for all $0\leq j\leq J+1$. In particular $$\Phi_x^{I_d+n_{d+1}+3(J+1)+1}(D_{\tau}(b))=D_{\delta_{J+1}}(R_b^{I_d+n_{d+1}+3(J+1)+1}(x))\subset\Psi_{m_{d+1}-(J+1)}$$

Therefore, since $2\delta_{J+1}\rho_{m_{d+1}-J-1}^2<\tau$, then $$\Phi_x^{I_d+n_{d+1}+3(J+1)+2}(D_{\tau}(b))=D_\tau(R_b^{I_d+n_{d+1}+3(J+1)+2}(x))\subset D_\frac{1}{\sqrt{2}}(1).$$

Following the analysis given by Lemma \ref{LPhi} we have that $$\Phi_x^{I_d+n_{d+1}+3j+2}(D_{\tau}(b))=D_\tau(R_b^{I_d+n_{d+1}+3j+2}(x)),$$for all $j\geq J+1$. In particular, for $j=m_{d+1}-1$ we have that $I_{d+1}=I_d+n_{d+1}+3(m_{d+1}-1)+2$ and $$\Phi_x^{I_{d+1}}(D_{\tau}(b))=D_\tau(R_b^{I_{d+1}}(x)).$$
\end{proof}

\subsection{Finding the parameter $b_{k+1}$}

In this subsection we give an algorithm to find the parameter $b_{k+1}$ from the parameter $b_k$, this algorithm is given in the proof of following lemma.

\begin{lemma}
Let $x\in D_\frac{1}{\sqrt{2}}(1)$ with strict itinerary $n_1m_1n_2m_2\ldots n_km_k$ under $R_{b_k}(z)$. Then, there exists a parameter $b_{k+1}\in D_{\tau_k}(b_k)$ such that $x$ has itinerary $n_1m_1n_2m_2\ldots n_km_kn_{k+1}m_{k+1}$ under $R_{b_{k+1}}(z)$.
\end{lemma}
\begin{proof}
Since $x$ is a point with strict itinerary $n_1m_1n_2m_2\ldots n_km_k$ under $R_{b_k}$, then by Proposition \ref{PTau}, $\Phi_x^{I_k}(D_\tau(b_k))=D_\tau(1)$ for all $\tau\leq\tau_k$. 

The above implies that, since $\frac{1}{2^{n_{k+1}}\sqrt{2}}<\tau_k$, then $\Phi_x^{I_k}(\mathbb{S}_\frac{1}{2^{n_{k+1}}\sqrt{2}}(b_k))=\mathbb{S}_\frac{1}{2^{n_{k+1}}\sqrt{2}}(1)$

2. in Lemma \ref{LPhi} implies that $\Phi_x^{I_k+1}(\mathbb{S}_\frac{1}{2^{n_{k+1}}\sqrt{2}}(b_k))=\mathbb{S}_\frac{1}{2^{n_{k+1}}\sqrt{2}}(0)$. Then by 3. in Lemma \ref{LPhi} we have that $\Phi_x^{I_k+n_{k+1}+1}(\mathbb{S}_\frac{1}{2^{n_{k+1}}\sqrt{2}}(b_k))=\mathbb{S}_\frac{1}{\sqrt{2}}(0)$. Therefore, there exists $\beta\in\mathbb{S}_\frac{1}{2^{n_{k+1}}\sqrt{2}}(b_k)$ such that $\Phi_x^{I_k+n_{k+1}+1}(\beta)=c_1^{b_k}$ and hence $$\Phi_x^{I_k+n_{k+1}+1}(D_\frac{1}{2^{n_{k+1}}\sqrt{2}}(\beta))=D_\frac{1}{\sqrt{2}}(c_1^{b_k}).$$Moreover, $$\Phi_x^{I_k+n_{k+1}+1}(\mathbb{S}_\frac{\rho_{m_{k+1}}}{2^{n_{k+1}}}(\beta))=\Psi_{m_{k+1}}.$$

Lemma \ref{LPhi} implies that $$\Phi_x^{I_k+n_{k+1}+3j+1}(\mathbb{S}_\frac{\rho_{m_{k+1}}}{2^{n_{k+1}}}(\beta))=\Psi_{m_{k+1}-j},$$for all $1\leq j< m_{k+1}$. In particular, when $j=m_{k+1}-1$ we have $$\Phi_x^{I_k+n_{k+1}+3(m_{k+1}-1)+1}(\mathbb{S}_\frac{\rho_{m_{k+1}}}{2^{n_{k+1}}}(\beta))=\Psi_1.$$

Therefore, there exists $\beta_2\in\mathbb{S}_\frac{\rho_{m_{k+1}}}{2^{n_{k+1}}}(\beta)$ such that $$\Phi_x^{I_k+n_{k+1}+3(m_{k+1}-1)+1}(\beta_2)=w_1^{b_1},$$and hence

%On the one hand, since $I_{k+1}=I_k+n_{k+1}+3(m_{k+1}-1)+2$, then by 1. in Lemma \ref{LPhi} $$\Phi_{I_{k+1}}^x(D_\frac{\rho_{m_{k+1}}}{2^{n_{k+1}}}(\beta_2))=D_\frac{1}{2\sqrt{2}}(R_{\beta_2}^{I_{k+1}}(x)).$$

%On the other hand, since $\Phi_{I_k+n_{k+1}+3(m_{k+1}-1)+1}^x(D_\frac{\rho_{m_{k+1}}}{2^{n_{k+1}}}(\beta_2))=D_\frac{1}{2}(w_1^b)$, then $|R_{\beta_2}^{I_k+n_{k+1}+3(m_{k+1}-1)+1}(x)-w_1^b|<\frac{1}{2}$. By Lemma \ref{Lw1w2} $|w_1^{\beta_2}-w_1^b|\leq\frac{1}{2\sqrt{2}}$, therefore $$|R_{\beta_2}^{I_k+n_{k+1}+3(m_{k+1}-1)+1}(x)-w_1^{\beta_2}|<\frac{1}{2}.$$

%The above implies that $|R_{\beta_2}^{I_{k+1}}(x)-1|<\frac{1}{2\sqrt{2}}$ and hence $D_\frac{1}{2\sqrt{2}}(R_{\beta_2}^{I_{k+1}}(x))=D_\frac{1}{2\sqrt{2}}(1)$, that is 

$$\Phi_x^{I_{k+1}}(D_\frac{\rho_{m_{k+1}}}{2^{n_{k+1}}}(\beta_2))=D_\frac{1}{2\sqrt{2}}(1).$$Therefore, there exists $b_{k+1}\in D_\frac{\rho_{m_{k+1}}}{2^{n_{k+1}}}(\beta_2)$ such that $$\Phi_x^{I_{k+1}}(b_{k+1})=1.$$

With this, we conclude that $x$ has strict itinerary $n_1m_1\cdots n_km_kn_{k+1}m_{k+1}$ under $R_{b_{k+1}}$. 
\end{proof}

By induction, using the previous lemma, we can find the parameters $b_2$, $b_3$, etcetera.

\begin{remark}
\label{Rbk}
Note that
\begin{itemize}
\item $b_{k+1}\in D_\frac{\rho_{m_{k+1}}}{2^{n_{k+1}}}(\beta_2)\subset\mathbb{S}_\frac{\rho_{m_{k+1}}}{2^{n_{k+1}}}(\beta)\subset B_\frac{1}{2^{n_{k+1}}\sqrt{2}}(b_k)\subset D_{\tau_k}(b_k)\subset B_{\tau_k}(b_k).$
\item $x$ has itinerary $n_1m_1\cdots n_km_k$ under $R_b$ for all $b\in D_\frac{\rho_{m_{k+1}}}{2^{n_{k+1}}}(\beta_2)$.
\item Since $\tau_{k+1}<\tau_k$ and $b_{k+1}\in D_{\tau_k}(b_k)$ then $D_{\tau_{k+1}}(b_{k+1})\subset D_{\tau_k}(b_k)$. Therefore, if $b\in D_{\tau_{k+1}}(b_{k+1})$ then $b\in D_{\tau_i}(b_i)$ for all $i=1,\ldots,k$.
\item Since $\tau_{k+1}<\tau_k$ and $b_{k+1}\in D_{\tau_k}(b_k)$ then $B_{\tau_{k+1}}(b_{k+1})\subset B_{\tau_k}(b_k)$.
\end{itemize}
\end{remark}

\subsection{Proof of Theorem A}

In this subsection we use previous lemmas and propositions to give a proof for the Theorem A. 

Note that $\lim_{k\to\infty}\tau_k=0$, therefore the intersection $\displaystyle{\bigcap_{k\geq1}}B_{\tau_k}(b_k)$ consists of a unique point, this point will be the desired parameter.   
 
%\begin{tmaa*}
%Let $\mathbb{C}_K$ be a complete and algebraically closed non-archimedean field with residual characteristic $\car(\widetilde{\mathbb{C}}_K)=2$. Then there exist constants $a,b\in\mathbb{C}_K$ such that the rational function $$R(z)=\frac{z^2-z}{bz-\frac{1}{a}}$$has at least a wandering component in its Fatou set.
%\end{tmaa*}
\begin{proof}[Proof of Theorem A]
Let $R_b(z)=\frac{z^2-z}{bz-\frac{1}{a}}$ be the rational function where $a$ is the parameter of the rational function in \eqref{Ecb1} and $$\displaystyle{b=\bigcap_{k\geq1}}B_{\tau_k}(b_k).$$We will prove that $R_b$ contains wandering components in its Fatou set.

By Remark \ref{Rbk}, $b\in D_{\tau_k}(b_k)$ for all $k\geq1$, this implies that $x$ has itinerary $n_1m_1n_2m_2\ldots$ under $R_b$. 

Let us consider the disk $D_r(x)\subset D_\frac{1}{\sqrt{2}}(1)$ with $r<\frac{\rho_{m_1}^2\sqrt{2}}{2^{n_1}}$. We will prove that for $j\geq1$, 
\begin{center}
$R^{I_j}(D_r(x))=D_{r_j}(R^{I_j}(x))$ with $r_j<\frac{\rho_{m_{j+1}}^2\sqrt{2}}{2^{n_{j+1}}}$,
\end{center}
(where $I_j$ is the number defined in \eqref{I}) and therefore $R^j(D_r(x))$ is a disk for all $j\geq1$. Moreover, the disk $D_r(x)$ has itinerary $n_1m_1n_2m_2\ldots$ 

We prove this by induction. First, we prove that $R^{I_1}(D_r(x))=D_{r_1}(R^{I_1}(x))$, with $r_1<\frac{\sqrt{2}\rho_{m_2}^2}{2^{n_2}}$. Since $D_r(x)\subset D_\frac{1}{\sqrt{2}}(1)$ then, by 2. in Lemma \ref{LRb}, $R(D_r(x))=D_r(R(x))$. Then, by 3. in Lemma \ref{LRb} we have that $$R^{n_1+1}(D_r(x))=D_{2^{n_1}r}(R^{n_1+1}(x)).$$ 

By hypothesis $2^{n_1}r<\rho_{m_1}^2\sqrt{2}$, then $R^{n_1+1}(D_r(x))\subset\Psi_{m_1}$ and moreover, by 1. in Lemma \ref{LRb} $$R^{n_1+2}(D_r(x))=D_{2^{n_1}r2\rho_{m_1}^2}(R^{n_1+2}(x)).$$

By 2. and 3. in Lemma \ref{LRb} we have that $R^{n_1+4}(D_r(x))=D_{2^{n_1}r2^2\rho_{m_1}^2}(R^{n_1+4}(x))$. Note that $$2^{n_1}r2^2\rho_{m_1}^2\leq2^{n_1}r<\sqrt{2}\rho_{m_1}^2<\sqrt{2}\rho_{m_1-1},$$therefore $D_{2^{n_1}r2^2\rho_{m_1}^2}(R^{n_1+4}(x))\subset\Psi_{m_1-1},$ and hence 

%Note that, since $\rho_i\leq\frac{1}{2}$ for all $i\geq1$, if $\delta<\sqrt{2}\rho_i^2$ then $$2^2\rho_i^2\delta\leq\delta<\sqrt{2}\rho_i^2<\sqrt{2}\rho_{i-1}^2.$$

%For the previous analysis and by 1. in Lemma \ref{LRb} 

$$R^{n_1+5}(D_r(x))=D_{2^{n_1}r2^3\rho_{m_1}^2\rho_{m_1-1}^2}(R^{n_1+5}(x)).$$By induction, using Lemma \ref{LRb} we conclude that $$R^{I_1}(D_r(x))=D_{r_1}(R^{I_1}(x)),$$with $r_1=2^{n_1}r2^{2m_1}\rho_{m_1}^2\rho_{m_1-1}^2\cdots\rho_1^2$.

By construction of the sequence $n_1m_1n_2m_2\ldots$, we have that $$2^{2m_1}\rho_{m_1}^2\rho_{m_1-1}^2\cdots\rho_1^2<\frac{1}{2^{n_2+1}}.$$Since $\rho_{m_2}^2>\frac{1}{8}$ and $r<\frac{\rho_{m_1}^2\sqrt{2}}{2^{n_1}}\leq\frac{1}{2^{n_1+1}\sqrt{2}}$, then $$r_1=2^{n_1}r2^{2m_1}\rho_{m_1}^2\rho_{m_1-1}^2\cdots\rho_1^2<\frac{2^{n_1}r}{2^{n_2+1}}<\frac{2^{n_1}\frac{1}{2^{n_1+1}\sqrt{2}}}{2^{n_2+1}}=\frac{1}{2^{n_2}4\sqrt{2}}<\frac{\sqrt{2}\rho_{m_2}^2}{2^{n_2}}.$$

Now, using induction, we prove that the affirmation is true for all $k\geq1$. Let us assume that $R^{I_k}(D_r(x))=D_{r_k}(R^{I_k}(x))$ with $r_k=2^{n_k}r_{k-1}2^{2m_k}\rho_{m_k}^2\rho_{m_k-1}^2\cdots\rho_1^2$ and $r_k<\frac{\sqrt{2}\rho_{m_{k+1}}^2}{2^{n_{k+1}}}$.

Since $R^{I_k}(D_r(x))=D_{r_k}(R^{I_k}(x))\subset D_\frac{1}{\sqrt{2}}(1)$, then using 2. and 3. in Lemma \ref{LRb} $$R^{I_k+n_{k+1}+1}(D_r(x))=D_{2^{n_{k+1}}r_k}(R^{I_k+n_{k+1}+1}(x))\subset\Psi_{m_{k+1}}.$$Note that, $2^{n_{k+1}}r_k<\sqrt{2}\rho_{m_{k+1}}^2$, therefore $$R^{I_k+n_{k+1}+2}(D_r(x))=D_{2^{n_{k+1}}r_k 2\rho_{m_{k+1}}^2}(R^{I_k+n_{k+1}+2}(x))\subset D_\frac{1}{\sqrt{2}}(1),$$and hence $$R^{I_k+n_{k+1}+4}(D_r(x))=D_{2^{n_{k+1}}r_k 2^2\rho_{m_{k+1}}^2}(R^{I_k+n_{k+1}+4}(x))\subset\Psi_{m_{k+1}-1}.$$

Since $r_k<\frac{\sqrt{2}\rho_{m_{k+1}}^2}{2^{n_{k+1}}}$ then $$2^{n_{k+1}}r_k2^{2j}\rho_{m_{k+1}}^2\cdots\rho_{m_{k+1}-j+1}^2<\sqrt{2}\rho_{m_{k+1}-j}^2,$$for all $j\in\{1,2,\ldots,m_{k+1}-1\}$. Therefore, by induction, using Lemma \ref{LRb} we conclude that $$R^{I_{k+1}}(D_r(x))=D_{r_{k+1}}(R^{I_{k+1}}(x)),$$with $r_{k+1}=2^{n_{k+1}}r_k2^{2(m_{k+1})}\rho_{m_{k+1}}^2\cdots\rho_1^2$.

Note that, since $r_k<\frac{\sqrt{2}\rho_{m_{k+1}}^2}{2^{n_{k+1}}}\leq\frac{1}{2^{n_{k+1}+1}\sqrt{2}}$, then
\begin{align*}
r_{k+1}&=2^{n_{k+1}}r_k2^{2(m_{k+1})}\rho_{m_{k+1}}^2\cdots\rho_1^2\\
       &<\frac{2^{n_{k+1}}r_k}{2^{n_{k+2}+1}}<\frac{2^{n_{k+1}}\frac{1}{2^{n_{k+1}+1}\sqrt{2}}}{2^{n_{k+2}+1}}\\
       &=\frac{1}{2^{n_{k+2}}4\sqrt{2}}<\frac{\sqrt{2}\rho_{m_{k+2}}^2}{2^{n_{k+2}}}.
\end{align*}

By induction, $R^{I_j}(D_r(x))=D_{r_j}(R^{I_j}(x))$ with $r_j<\frac{\rho_{m_{j+1}}^2\sqrt{2}}{2^{n_{j+1}}}$. Since $D_r(x)$ has itinerary $n_1m_1n_2m_2,\ldots$ and $n_1>n_2>\cdots$ then $D_r(x)$ is contained in a wandering component. 
\end{proof}

\end{document}